\documentclass[12pt]{amsart}

\usepackage[english]{babel}
\usepackage[cp1250]{inputenc}
\usepackage{graphicx}
\usepackage{url}
\usepackage{amsmath}
\usepackage{amsthm}
\usepackage{amsfonts}
\usepackage[all]{xy}
\usepackage{multicol}
\usepackage{hyperref}
\usepackage{geometry}
\usepackage{pinlabel}
\usepackage{tikz-cd}

\allowdisplaybreaks

\newtheorem{Theorem}{Theorem}[section]
\newtheorem{lemma}[Theorem]{Lemma}
\newtheorem{proposition}[Theorem]{Proposition}
\newtheorem{corollary}[Theorem]{Corollary}
\newtheorem{definition}[Theorem]{Definition}
\newtheorem{remark}[Theorem]{Remark}
\newtheorem{example}[Theorem]{Example}

\newcommand{\ZZ}{\mathbb {Z}}

\newcommand{\tr}{\rhd}
\newcommand{\tl}{\lhd}

\newcommand{\wh}{\widehat }

\begin{document}
\title{Knot symmetries and the fundamental quandle}

\author{Eva Horvat}
\address{University of Ljubljana\\
Faculty of Education\\
Kardeljeva plo\v s\v cad 16\\
1000 Ljubljana, Slovenia}
\email{eva.horvat@pef.uni-lj.si}

\date{\today}
\begin{abstract}

We establish a relationship between the knot symmetries and the automorphisms of the knot quandle. We identify the homeomorphisms of the pair $(S^{3},K)$ that induce the (anti)automorphisms of the fundamental quandle $Q(K)$. We show that every quandle (anti)automorphism of $Q(K)$ is induced by a homeomorphism of the pair $(S^{3},K)$. As an application of those results, we are able to explore some symmetry properties of a knot based on the presentation of its fundamental quandle, which is easily derived from a knot diagram.  
\end{abstract}
\maketitle

\begin{section}{Introduction}\label{sec1}
It is well known that the knot quandle is a complete knot invariant \cite{MA}. The knot quandle and various derived quandle invariants have been extensively used to study and distinguish nonequivalent knots. Our idea is to study knot equivalences from the viewpoint of the fundamental quandle. In this paper, we establish a relationship between the knot symmetries and the automorphisms of the knot quandle. Our goal is to investigate the complex topological information, hidden in the group of knot symmetries, from a purely algebraical perspective of the knot quandle.     
Our main results are the following. 
\begin{proposition}  Let $f\colon (S^{3},K)\to (S^{3},K)$ be a homeomorphism for which $[f|_{\partial N_{K}}]=\pm 1\in MCG(T^{2})$. Then $f$ induces a map $f_{*}\colon Q(K)\to Q(K)$ that is either a quandle automorphism or a quandle antiautomorphism.
\end{proposition}
\begin{proposition}  Let $F\colon Q(K)\to Q(K)$ be an (anti)automorphism of the fundamental quandle of a nontrivial knot $K$. Then $F$ is induced by a homeomorphism that preserves the orientation of the normal bundle if $F$ is an automorphism, and reverses the orientation of the normal bundle if $F$ is an antiautomorphism. 
\end{proposition}
This paper is organized as follows. In the Section \ref{sec2}, we introduce the basic concepts from the theory of knot quandles that will be needed in the rest of the paper. Subsection \ref{subs21} contains the definition of the quandle, the augmented quandle, the associated group and some basics about the quandle homomorphisms. In the Subsection \ref{subs22} we define the quandle presentations. In the Subsection \ref{subs23}, we define the fundamental quandle of a knot in $S^{3}$, note some of its properties and describe its presentation. The Section \ref{sec3} contains the core results of the paper. Recalling the basics of knot symmetries, we first establish how and when a knot symmetry induces a knot quandle (anti)automorphism. Then we use a Theorem of Matveev to show that each knot quandle (anti)automorphism is induced by a knot symmetry. In the final Section \ref{sec4}, we do some calculations displaying the use of our results in investigating the symmetries of knots. We provide a \verb|Phyton| code for calculating the $Q$-group of symmetries from an alternating planar diagram of a knot. 
\end{section}

\begin{section}{Preliminaries}\label{sec2}

\begin{subsection}{The definition of a quandle}\label{subs21}
\begin{definition}\label{def1}\cite{JJ}
A \textbf{quandle} is an algebraic structure comprising a nonempty set $Q$ with two binary operations $(x,y)\mapsto x\tr y$ and $(x,y)\mapsto x\tl y$, which satisfies three axioms:
\begin{enumerate}
\item $x\tr x=x$.
\item $(x\tr y)\tl y=x=(x\tl y)\tr y$. 
\item $(x\tr y)\tr z=(x\tr z)\tr (y\tr z)$.   
\end{enumerate} 
\end{definition}
It follows from the second axiom that the map $S_{y}\colon Q\to Q$ defined by $S_{y}(x)=x\tr y$ is a bijection for any $y\in Q$, and by the third axiom it is an automorphism of $Q$. By the first quandle axiom the automorphism $S_{y}$ fixes $y$. $S_{y}$ is called an \textbf{inner automorphism} of $Q$. The subgroup $Inn(Q)$ of all the inner automorphisms is a normal subgroup of $Aut(Q)$ \cite{JJ}.

\begin{example} \label{ex1} Let $G$ be a group. Define the operations $\tr , \tl \colon G\times G\to G$ by $a\tr b=b^{-1}ab$ and $a\tl b=bab^{-1}$ for any $a,b\in G$. It is easy to see that these operations satisfy all the quandle axioms, and the resulting quandle is denoted by $G_{conj}$. Thus we obtain a quandle from any group by forgetting multiplication and setting both conjugations as the operations.  
\end{example}

\begin{example} \label{ex2} For any quandle $Q=(S,\tr ,\tl )$, the triple $(S,\tl ,\tr )$ defines another quandle, which we denote by $Q^{d}$ and is called the \textbf{dual quandle} of $Q$. 
\end{example} 

\begin{definition} \label{def2} Let $G$ be a group, acting on itself by conjugation as $g^{h}:=h^{-1}gh$ for any $g,h\in G$. An \textbf{augmented quandle} $(X,G)$ is a set $X$ with an action by the group $G$, written as $(x,g)\mapsto x^{g}$ and a function $\partial \colon X\to G$, satisfying the augmentation conditions 
\begin{enumerate}
\item $x^{\partial x}=x$ for all $x\in X$,
\item $\partial (x^{g})=g^{-1}(\partial x)g$ for all $x\in X, g\in G$.
\end{enumerate} The quandle operations on $X$ are then defined by $x\tr y:=x^{\partial y}$ and $x\tl y:=x^{(\partial y)^{-1}}$.
\end{definition}

\begin{definition} \label{def3} Let $Q_{1}$ and $Q_{2}$ be quandles.

 A \textbf{quandle homomorphism} from $Q_{1}$ to $Q_{2}$ is a map $f\colon Q_{1}\to Q_{2}$ satisfying $f(x\tr y)=f(x)\tr f(y)$ for any $x,y\in Q_{1}$. 

A \textbf{quandle antihomomorphism} from $Q_{1}$ to $Q_{2}$ is a map $g\colon Q_{1}\to Q_{2}$ satisfying $g(x\tr y)=g(x)\tl g(y)$ for any $x,y\in Q_{1}$.
\end{definition}

\begin{remark} \label{rem1} Observe that a quandle antihomomorphism $g\colon Q_{1}\to Q_{2}$ is actually a quandle homomorphism from $Q_{1}$ to $Q_{2}^{d}$. 
\end{remark}

\begin{lemma} \label{lemma1} If $f\colon Q_{1}\to Q_{2}$ is a quandle homomorphism, then $f(x\tl y)=f(x)\tl f(y)$. 

 If $f\colon Q_{1}\to Q_{2}$ is a quandle antihomomorphism, then $f(x\tl y)=f(x)\tr f(y)$. 
\end{lemma}
\begin{proof}Let $f\colon Q_{1}\to Q_{2}$ be a quandle homomorphism and choose $x,y\in Q_{1}$. Denoting $w=x\tl y$, we compute $w\tr y=x\Rightarrow f(w)\tr f(y)=f(x)\Rightarrow f(x\tl y)=f(w)=f(x)\tl f(y)$. A similar proof settles the case when $f$ is a quandle antihomomorphism. 
\end{proof}

\begin{lemma} \label{lemma2} For any quandle $Q$, the set $$Aut ^{C}(Q)=\{f\colon Q\to Q|\, f\textrm{ a bijection}, f\textrm{ a homomorphism or an antihomomorphism}\}$$ with the composition operation forms a group. 
\end{lemma}
\begin{proof} The set $Sym(Q)$ of all bijective maps of the set $Q$ to itself forms a group for composition. We would like to show that $Aut^{C}(Q)\subset Sym(Q)$ is a subgroup. Thus, we need to show that $Aut^{C}(Q)$ is closed under composition and invertation. A composition of two quandle automorphisms is obviously a quandle automorphism. Now let $g_{1},g_{2}\in Aut^{C}(Q)$ be two antihomomorphisms, and choose any $x,y\in Q$. We may compute:
\begin{xalignat*}{1}
& x\tr y=z\quad \Rightarrow \quad g_{2}(x)\tl g_{2}(y)=g_{2}(z)\quad \Rightarrow \quad g_{2}(z)\tr g_{2}(y)=g_{2}(x)\\
& \Rightarrow \quad g_{1}(g_{2}(z))\tl g_{1}(g_{2}(y))=g_{1}(g_{2}(x))\quad \Rightarrow \quad g_{1}(g_{2}(x))\tr g_{1}(g_{2}(y))=g_{1}(g_{2}(z))\;,
\end{xalignat*} thus $g_{1}\circ g_{2}$ is a quandle automorphism. In a similar way we show that a composition of a quandle homomorphism and a quandle antihomomorphism (or vice versa) is a quandle antihomomorphism.  

The inverse of a quandle automorphism is a quandle automorphism. The inverse of any quandle antihomomorphism $g$ is a quandle antihomomorphism, since their composition is the identity automorphism:
 \begin{xalignat*}{1}
& x\tr y=z\quad \Rightarrow \quad g(x)\tl g(y)=g(z)\quad \Rightarrow \quad g(z)\tr g(y)=g(x)\quad \Rightarrow \quad\\
& g^{-1}(g(z)\tr g(y))=x=z\tl y\quad \Rightarrow \quad g^{-1}(g(z)\tr g(y))=g^{-1}(g(z))\tl g^{-1}(g(y))\;.
\end{xalignat*}
\end{proof}

\begin{definition} \label{def4} Let $Q$ be a quandle. The \textbf{associated group} $\textrm{As}(Q)$ of the quandle $Q$ is defined as $\textrm{As}(Q)=F(Q)/K$, where $F(Q)$ is the free group, generated by the set $Q$, and $K$ is the normal subgroup of $F(Q)$, given by $K=\{(a\tr b)b^{-1}a^{-1}b,\, (a\tl b)ba^{-1}b^{-1}|\, a,b\in Q\}$. 
\end{definition}

\begin{lemma}\label{lemma3} Let $Q$ be a quandle and let $\eta \colon Q\to As(Q)$ be the natural map. Let $G$ be a group with the conjugation quandle $G_{conj}$ and the natural map $\partial \colon G_{conj}\to G$. Given any quandle homomorphism $f\colon Q\to G_{conj}$, there exists a unique group homomorphism $f_{\# }\colon As(Q)\to G$ such that $f_{\# }\circ \eta =\partial \circ f$. 
\[
  \begin{tikzcd}
   Q \arrow{r}{\eta} \arrow[swap]{d}{f} & As(Q) \arrow[dashed]{d}{f_{\# }}\\
    G_{conj} \arrow{r}{\partial } & G
 \end{tikzcd}
\]
\end{lemma}
\begin{proof} Let $As(Q)=F(Q)/K$, where $K$ is the normal subgroup from the Definition \ref{def4}. The quandle homomorphism $f$ induces a map $\phi \colon F(Q)\to G$. By the definition of a quandle homomorphism, the quandle operations of the Example \ref{ex1} yield $\phi (a\tr b)=\phi (b)^{-1}\phi (a)\phi (b)$ and $\phi (a\tl b)=\phi (b)\phi (a)\phi (b)^{-1}$ for every $a,b\in F(Q)$, thus the normal subgroup $K$ lies in the kernel of $\phi $. It follows that $\phi $ factors through a unique group homomorphism $f_{\# }\colon As(Q)\to G$.
\end{proof}

\end{subsection}

\begin{subsection}{Quandle presentations}\label{subs22}
We recall the following from \cite{MAN}. Let $A$ be a set we call the \textit{alphabet}, whose elements we call the \textit{letters}. A \textit{word} in $A$ is any finite sequence, consisting of letters and the symbols $(, ), \tr $ and $\tl $. Define inductively the set $\mathcal{D}(A)$ of \textit{admissible words} according to the rules:
\begin{enumerate}
\item For any $a\in A$, the word $a$ is admissible,
\item If $w$ and $z$ are admissible words, then $w\tr z$ and $w\tl z$ are also admissible,
\item There are no other admissible words except those obtained inductively by the rules (1) and (2). 
\end{enumerate}
Let $R$ be a set of \textit{relations}; these are identities of the form $r_{i}=s_{i}$, where $r_{i},s_{i}\in \mathcal{D}(A)$. Define an equivalence relation $\sim $ on the set $\mathcal{D}(A)$ as follows: $w_{1}\sim w_{2}$ if and only if $w_{1}$ can be transformed into $w_{2}$ by some sequence of the rules (1) - (5) described below: 
\begin{enumerate}
\item $x\tr x \Leftrightarrow x$; 
\item $(x\tr y)\tl y \Leftrightarrow x$;
\item $(x\tl y)\tr y \Leftrightarrow x$;
\item $(x\tr y)\tr z \Leftrightarrow (x\tr z)\tr (y\tr z)$;
\item $r_{i} \Leftrightarrow s_{i}$.
\end{enumerate}  
The set of equivalence classes $\mathcal{D}(A)/_{\sim }$ is denoted by $\Gamma \langle A|R\rangle $. It is easy to show that $\Gamma \langle A|R\rangle $ is a quandle, and we call $\langle A|R\rangle $ the \textit{presentation} of this quandle.
\end{subsection}

\begin{subsection}{The fundamental quandle of a knot}\label{subs23}
We briefly summarize the following from \cite{MA} for the reader's convenience. Let $K$ be a knot in $S^{3}$ with a fixed orientation of its normal bundle. Denote by $N_{K}$ the regular neighborhood of $K$ and let $E_{K}=S^{3}-\textrm{Int}(N_{K})$.  Choose a basepoint $z_{K}\in E_{K}$, a point $z'_{K}\in \partial N_{K}$ and a path $s_{K}\subset E_{K}$ from $z'_{K}$ to $z_{K}$. Denote by $G_{K}:=\pi _{1}(E_{K},z_{K})$ the fundamental group of the knot. The path $s_{K}$ defines an inclusion of $\pi _{1}(\partial N_{K},z'_{K})$ into $G_{K}$ by the formula $\widehat{a}\mapsto [\overline{s}_{K}\cdot a\cdot s_{K}]$, where $[n]=\widehat{n}$. The image of this inclusion is the peripheral subgroup $H_{K}\leq G_{K}$. Denote $$\Gamma _{K}=\{\textrm{homotopy classes of paths in $E_{K}$ from a point in $\partial N_{K}$ to $z_{K}$}\}\;.$$ During the homotopy the initial point may move around on $\partial N_{K}$, while the final point is kept fixed. The group $G_{K}$ acts on the set $\Gamma _{K}$ by $\widehat{a}^{\widehat{g}}:=[a\cdot g]$. 

Using this action, the set $\Gamma _{K}$ may be equipped with a structure of an augmented quandle. Any $p\in \partial N_{K}$ lies on a unique meridian circle of the normal circle bundle and we denote by $m_{p}$ the loop based at $p$ which follows around the meridian in the positive direction. The image of $m_{z'}$ in $H_{K}$ is denoted by $m_{K}$ and is called a \textbf{meridian} of $K$. 

\begin{definition}\label{def5}The \textbf{fundamental quandle} $Q(K)$ of the knot $K$ is the augmented quandle $(\Gamma _{K},G_{K})$ as above with the function $\partial \colon \Gamma _{K}\to G_{K}$ defined as follows. Given two classes $\widehat{a},\widehat{b}\in \Gamma _{K}$, which are represented by the paths $a$ and $b$ respectively, define $\partial (\widehat{b})=[\overline{b}\cdot m_{b(0)}\cdot b]$, where $\cdot $ denotes concatenation of paths. This produces the quandle operations 
\begin{xalignat*}{1}
& \widehat{a}\tr \widehat{b}=\widehat{a} ^{\partial (\widehat{b})}=[a\cdot \overline{b}\cdot m_{b(0)}\cdot b]\textrm{  and  }\widehat{a}\tl \widehat{b}=\widehat{a} ^{\partial (\wh{b})^{-1}}=[a\cdot \overline{b}\cdot \overline{m}_{b(0)}\cdot b]
\end{xalignat*}
\end{definition}

\begin{remark} \label{rem2} For a knot $K$ with a fixed orientation of its normal bundle, denote by $K^{d}$ the same knot with the opposite orientation of the normal bundle. It follows that $Q(K^{d})$ is the dual quandle of $Q(K)$, as defined in the Example \ref{ex2}. 
\end{remark}

\begin{lemma} \label{lemmax} The function $\partial \colon \Gamma _{K}\to G_{K}$ is a bijection.
\end{lemma}
\begin{proof} Let $\partial \wh{a}=\partial \wh{b}$ for two elements $\wh{a},\wh{b}\in \Gamma _{K}$. It follows that there exists a homotopy $h_{t}\colon [0,1]\to E_{K}$, such that $h_{0}=\overline{a}m_{a(0)}a$ and $h_{1}=\overline{b}m_{b(0)}b$. Using a suitable reparametrization, we may assume that $h|_{[\frac{2}{3},1]}$ is a homotopy between the paths $a$ and $b$, thus $\wh{a}=\wh{b}$. 

To see that $\partial $ is surjective, observe that the elements $\partial x$ for $x\in \Gamma _{K}$ are exactly the generators of $G_{K}$ in the Wirtinger presentation of the knot group. 
\end{proof} 

\begin{lemma}\cite[Lemma 2]{MA} \label{lemma4} For any knot $K$, the action of the group $G_{K}$ on $Q(K)$ is transitive. The stabilizer subgroup of the element $\widehat{s}_{K}=[s_{K}]\in Q(K)$ coincides with $H_{K}$. 
\end{lemma}
\begin{proof} Choose two paths $a$ and $b$ which represent the homotopy classes $\widehat{a},\widehat{b}\in Q(K)$. Choose a path $n\subset  \partial N_{K}$ from $a(0)$ to $b(0)$ and let $\widehat{g}=[\overline{a}\cdot n\cdot b]$. Then we have $\widehat{a}^{\widehat{g}}=[a\cdot \overline{a}\cdot n\cdot b]=[n\cdot b]=\widehat{b}$.  

By the definition of $H_{K}$, each element $\widehat{h}\in H_{K}$ has the form $\widehat{h}=[\overline{s}_{K}\cdot n\cdot s_{K}]$ for some loop $n$ in $\partial N_{K}$ starting and ending at $z'_{K}$. Thus, $\widehat{s}_{K}^{\widehat{h}}=\widehat{s}_{K}$ and $H_{K}$ acts trivially on $\widehat{s}_{K}$. On the other hand, if $\widehat{s}_{K}^{\widehat{g}}=\widehat{s}_{K}$ for some $\widehat{g}\in G_{K}$, then there is a homotopy $H(t,u)$ between the paths $H(0,u)=s_{K}\cdot g$ and $H(1,u)=s_{K}$. Let $H(t,0)=n\subset \partial N_{K}$ be the path, traced by the starting point of the path $s_{K}\cdot g$ under this homotopy. Then the path $\overline{s}_{K}\cdot \overline{n}\cdot s_{K}\cdot g$ is homotopic to a constant, which means that $\widehat{g}=[\overline{s}_{K}\cdot n\cdot s_{K}]$ lies in $H_{K}$.  
\end{proof} 

\begin{example}\label{ex3}\textbf{The presentation of the fundamental quandle of a knot.} 
Let $K$ be a knot, given by a knot diagram $D_{K}$. Label the arcs of $D_{K}$ by $y_{1},\ldots ,y_{n}$. Each crossing of the diagram consists of an overcrossing arc $y_{j}$ and two arcs $y_{i}$ and $y_{k}$ of the undercrossing strand. If we see the arc $y_{i}$ on the right (and $y_{k}$ on the left) of $y_{j}$ when passing along $y_{j}$, then we define the crossing relation $y_{i}\tr y_{j}=y_{k}$. Define a quandle $$\Gamma \langle y_{1},\ldots ,y_{n}|\, \textrm{crossing relations of $D_{K}$}\rangle \;.$$ It can be shown that $\Gamma \langle y_{1},\ldots ,y_{n}|\, \textrm{crossing relations of $D_{K}$}\rangle $ is isomorphic to the fundamental quandle $Q(K)$ \cite{MAN}. Here we offer a brief discussion of this correspondence. 

For each arc $y_{i}$ of the diagram $D_{K}$, choose a path $x_{i}$ from a point on $\partial N_{K}$, corresponding to the arc $y_{i}$, to the basepoint. Let $x_{i}$ be such that wherever its projection intersects the diagram $D_{K}$, it goes over the knot $K$. The homotopy classes of the paths $x_{1},\ldots ,x_{n}$ represent the elements of the set $\Gamma _{K}$. 
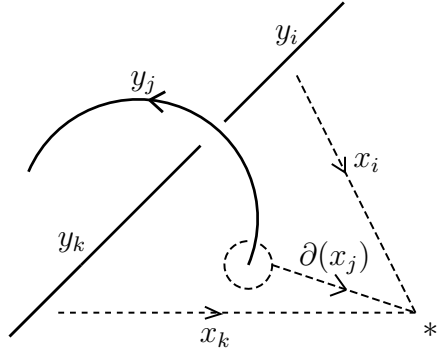
\begin{figure}[h]
\labellist
\normalsize \hair 2pt
\pinlabel $x_{k}$ at 260 0
\pinlabel $\partial (x_{j})$ at 410 100
\pinlabel $x_{i}$ at 450 220
\pinlabel $y_{k}$ at 80 120
\pinlabel $y_{j}$ at 170 320
\pinlabel $y_{i}$ at 350 380
\pinlabel $*$ at 530 10
\endlabellist
\begin{center}
\includegraphics[scale=0.3]{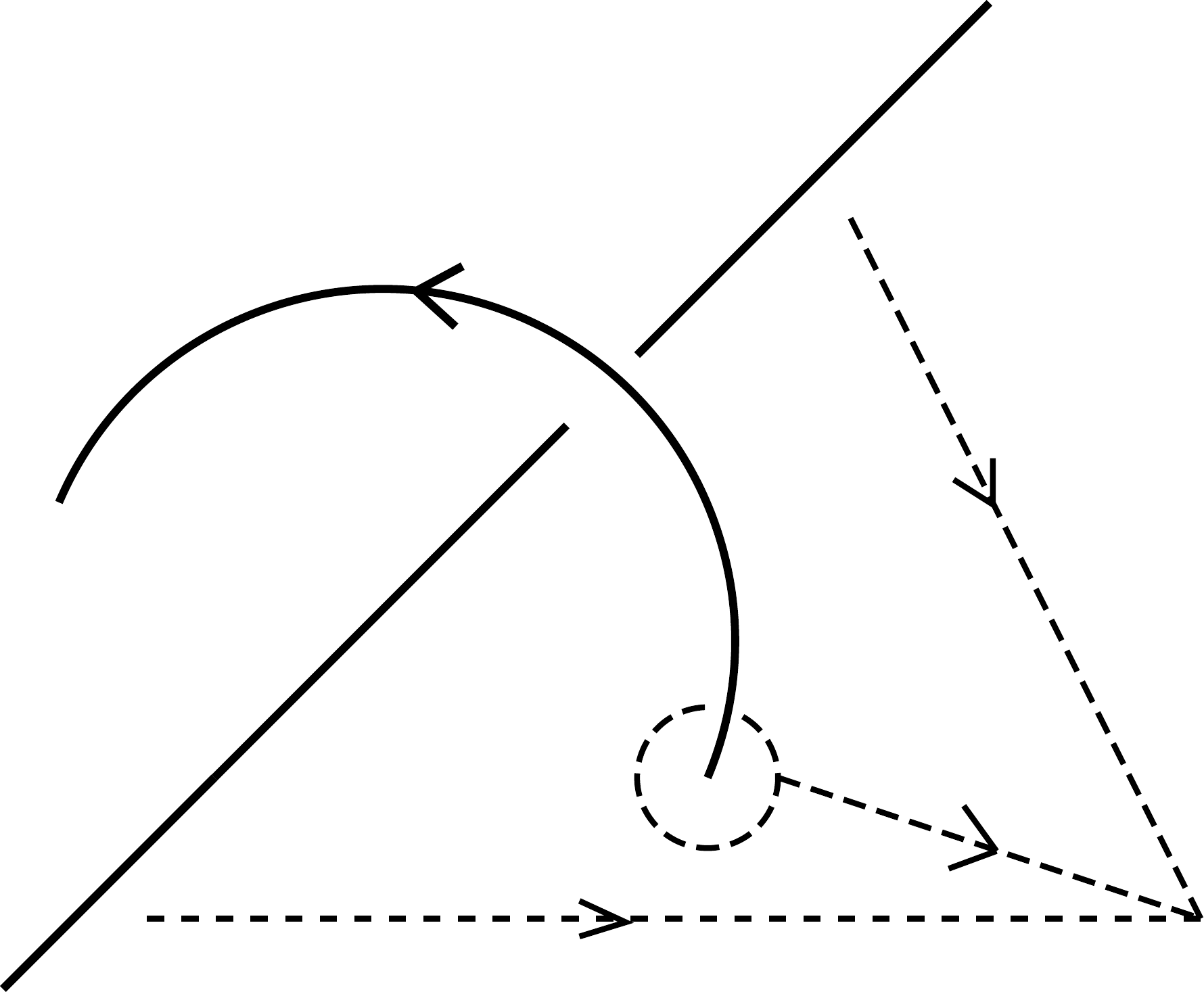}
\caption{The illustration of the crossing relation}
\label{fig:relation}
\end{center}
\end{figure}
The augmentation map $\partial \colon \Gamma _{K}\to G_{K}$ is given by $\partial (\wh{x}_{i})=[\overline{x}_{i}\cdot m_{x_{i}(0)}\cdot x_{i}]$ and the fundamental group $G_{K}$ is actually generated by the images $\partial (\wh{x}_{1}),\ldots ,\partial (\wh{x}_{n})$.
Consider a crossing of $D_{K}$ with the overcrossing arc $y_{j}$ and the undercrossing arcs $y_{i}$ and $y_{k}$. If we see the arc $y_{i}$ on the right (and $y_{k}$ on the left) of $y_{j}$ when passing along $y_{j}$, then we may see the homotopy between $x_{i}\cdot (\overline{x}_{j}\cdot m_{x_{j}(0)}\cdot x_{j})$ and $x_{k}$ in the Figure \ref{fig:relation}. This homotopy implies the crossing relation $\wh{x}_{i}\tr \wh{x}_{j}=\wh{x}_{k}$ in the fundamental quandle $Q(K)$.  
\end{example}

\begin{lemma} \label{lemma5} The fundamental group $G_{K}$ is the associated group $As(Q(K))$ of the fundamental quandle of the knot $K$. 
\end{lemma}
\begin{proof} Recall the Wirtinger presentation of the fundamental group of a knot. Using notation from the Example \ref{ex3}, the fundamental group $G_{K}$ is generated by the homotopy classes $\partial (\wh{x}_{1}),\ldots ,\partial (\wh{x}_{n})$. Therefore, the augmentation map $\partial \colon \Gamma _{K}\to G_{K}$ is a bijection between the generating sets of $Q(K)$ and $G_{K}$. In the associated group $As(Q(K))$, every crossing relation $\wh{x}_{i}\tr \wh{x}_{j}=\wh{x}_{k}$ is equivalent to the relation $\wh{x}_{i}=\wh{x}_{j}\wh{x}_{k}\wh{x}_{j}^{-1}$, which yields the Wirtinger relation $\partial (\wh{x}_{i})=\partial (\wh{x}_{j})\partial (\wh{x}_{k})\partial (\wh{x}_{j})^{-1}$ of this crossing in the fundamental group $G_{K}$.  
\end{proof}

\begin{remark} \label{rem3} Let $F\colon Q(K)\to Q(K)$ be a quandle automorphism of the fundamental quandle. We may specify how $F$ interfers with the action of $G_{K}$, defined above. Let $\widehat{a},\widehat{b}$ be two elements of $Q(K)$, represented by the paths $a$ and $b$. Let $\widehat{g}\in G_{K}$ be an element of the fundamental group. By the Lemma \ref{lemma5}, the element $\widehat{g}$ may be represented as $\widehat{g}=\partial \widehat{c}$ for some $\widehat{c}\in Q(K)$. By the Lemma \ref{lemma3}, the automorphism $F$ induces a unique group isomorphism $F_{\# }\colon G_{K}\to G_{K}$, for which $\partial F(\widehat{c})=F_{\# }(\widehat{g})$. Then we have
\[ F(\widehat{a}^{\widehat{g}})=F(\widehat{a}^{\partial \widehat{c}})=F(\widehat{a}\tr \widehat{c})=F(\widehat{a})\tr F(\widehat{c})=F(\widehat{a})^{\partial F(\widehat{c})}=F(\widehat{a})^{F_{\# }(\widehat{g})}\;.
\] 
\end{remark}

\end{subsection}
\end{section}

\begin{section}{Knot symmetries and the automorphisms of the knot quandle}\label{sec3}

Recall that the symmetry group of a knot $K$ in $S^{3}$ is the mapping class group of the pair $(S^{3},K)$, meaning the fourth term in the short exact sequence $$1\rightarrow Aut_{0}(S^{3},K)\rightarrow Aut(S^{3},K)\rightarrow MCG(S^{3},K)\rightarrow 1\;,$$ where $Aut(S^{3},K)$ denotes the group of homeomorphisms of the pair $(S^{3},K)$ and $Aut_{0}(S^{3},K)$ denotes the normal subgroup of those homeomorphisms which are isotopic to the identity. According to Hoste et al. \cite{HO}, there are four types of homeomorphisms in $Aut(S^{3},K)$: 
\begin{enumerate}
\item those which preserve the orientations of $K$ and of $S^{3}$,
\item those which reverse the orientation of $K$ and preserve the orientation of $S^{3}$,
\item those which preserve the orientation of $K$ and reverse the orientation of $S^{3}$,
\item those which reverse the orientation of $K$ and of $S^{3}$.
\end{enumerate}
For an oriented knot $K$ in $S^{3}$, denote by $rK$ the same knot with the opposite orientation, by $mK$ the mirror image of $K$ and by $rmK$ the mirror image of $K$ with the opposite orientation. 

\begin{center}
\begin{tabular}{|c|c|c|c|}
\hline 
\textbf{class} & \textbf{symmetries} & \textbf{knot symmetries} & \textbf{knot equivalences}\\
\hline
c & (1) & chiral, noninvertible & \quad \\
\hline
+ & (1), (3) & + amphichiral, noninvertible & $K=mK$\\ 
\hline
\-- & (1), (4) & - amphichiral, noninvertible & $K=rmK$\\
\hline
i & (1), (2) & chiral, invertible & $K=rK$\\
\hline
a & (1), (2), (3), (4) & + and - amphichiral, invertible & $K=rK=mK=rmK$\\
\hline
\end{tabular}
\end{center}

\begin{proposition} \label{prop1} Let $f\colon (S^{3},K)\to (S^{3},K)$ be a homeomorphism for which $[f|_{\partial N_{K}}]=\pm 1\in MCG(T^{2})$. Then $f$ induces a map $f_{*}\colon Q(K)\to Q(K)$ that is either a quandle automorphism or a quandle antiautomorphism.
\end{proposition}
\begin{proof} Let $f\colon (S^{3},K,z_{K})\to (S^{3},K,z_{K})$ be a homeomorphism. The underlying set $\Gamma _{K}$ of the fundamental quandle $Q(K)$ is actually the relative homotopy group $\pi _{1}(E_{K},\partial N_{K}\cup \{z_{K}\},z_{K})$. Since $f(K)=K$, it follows that $f(E_{K})\cong E_{K}$ and $f(\partial N_{K})\cong \partial N_{K}$, thus $f$ induces a map 
$$f_{*}\colon \Gamma _{K}=\pi _{1}(E_{K},\partial N_{K}\cup \{z_{K}\},z_{K})\to \pi _{1}(f(E_{K}),f(\partial N_{K})\cup \{f(z_{K})\},f(z_{K}))\cong \Gamma _{K}$$
on the relative homotopy group. Define $f_{*}(\wh {a})=[f\circ a]$, where $\wh{a}=[a]$. If $a$ and $a'$ are two different representatives of the class $\wh{a}\in \Gamma _{K}$, then there exists a homotopy $h_{t}\colon [0,1]\to E_{K}$ such that $h_{0}=a$ and $h_{1}=a'$. Then $f\circ h_{t}$ is a homotopy from $f_{*}(\wh{a})$ to $f_{*}(\wh{a'})$, thus $f_{*}$ is a well defined map on $\Gamma _{K}$. 

Let $\wh{a},\wh{b}\in \Gamma _{K}$ be represented by the respective paths $a$ and $b$. The quandle operation on $Q(K)$ is defined by $\wh{a}\tr \wh{b}=[a\cdot \overline{b}\cdot m_{b(0)}\cdot b]$, thus we have $$f_{*}(\wh{a}\tr \wh{b})=[f(a\cdot \overline{b}\cdot m_{b(0)}\cdot b)]=[f(a)\cdot f(\overline{b})\cdot f(m_{b(0)})\cdot f(b)]\;.$$ Now $f(m_{b(0)})$ is a loop in $f(\partial N_{K})$ based at $f(b(0))$. Since the restriction $f|_{\partial N_{K}}$ is isotopic to $\pm $ identity, the loop $f(m_{b(0)})$ also represents a meridian circle of $K$, which may be either $m_{f(b(0))}$ or $\overline{m}_{f(b(0))}$. It follows that 
\[ f_{*}(\wh{a}\tr \wh{b})=\begin{cases} f_{*}(\wh{a})\tr f_{*}(\wh{b}) &\textit{if f preserves the orientation of the normal bundle,}  \\
f_{*}(\wh{a})\tl f_{*}(\wh{b}) & \textit{if f reverses the orientation of the normal bundle,} \end{cases} \]
therefore the induced map $f_{*}\colon Q(K)\to Q(K)$ is a quandle (anti)homomorphism. 

Since $f$ is surjective, the induced map $f_{*}$ is a surjective homomorphism on $Q(K)$. To show the injectivity of $f_{*}$, let $f_{*}(\wh{a})=f_{*}(\wh{b})$ for two elements $\wh{a},\wh{b}\in Q(K)$. Choosing representatives $a$ and $b$ of the respective classes $\wh{a}$ and $\wh{b}$, there exists a homotopy $h_{t}\colon [0,1]\to E_{K}$ for which $h_{0}=f\circ a$ and $h_{1}=f\circ b$. Then $g_{t}=f^{-1}\circ h_{t}$ is a homotopy from $a$ to $b$, thus $\wh{a}=\wh{b}\in Q(K)$.   
\end{proof}

\begin{example} \label{ex4} Consider the knot $5_1$, given by the diagram of the Figure \ref{fig:5_1}. It is known that the knot $5_{1}$ is 5-periodic, thus there exists a nontrivial homeomorphism $f\colon (S^{3},5_1)\to (S^{3},5_1)$ of order 5. Let us find the corresponding quandle automorphism. The fundamental quandle has a presentation:
\[ Q(5_{1})=\left <a,b,c,d,e|\, b\tr a=e, c\tr b=a, d\tr c=b, e\tr d=c, a\tr e=d\right >\;.\] Consider the map $F\colon Q(5_1)\to Q(5_1)$, given by $F(a)=b$, $F(b)=c$, $F(c)=d$, $F(d)=e$ and $F(e)=a$. Since $F$ preserves the crossing relations of the above presentation, it defines a quandle automorphism of order 5. 
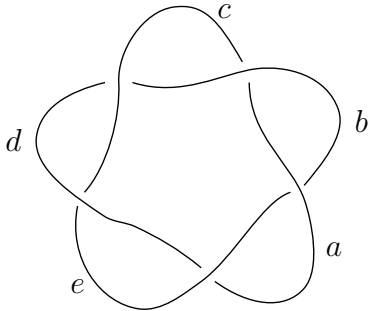
\begin{figure}[h]
\labellist
\normalsize \hair 2pt
\pinlabel $a$ at 170 40
\pinlabel $b$ at 185 110
\pinlabel $c$ at 110 170 
\pinlabel $d$ at -5 100
\pinlabel $e$ at 30 20
\endlabellist
\begin{center}
\includegraphics[scale=0.7]{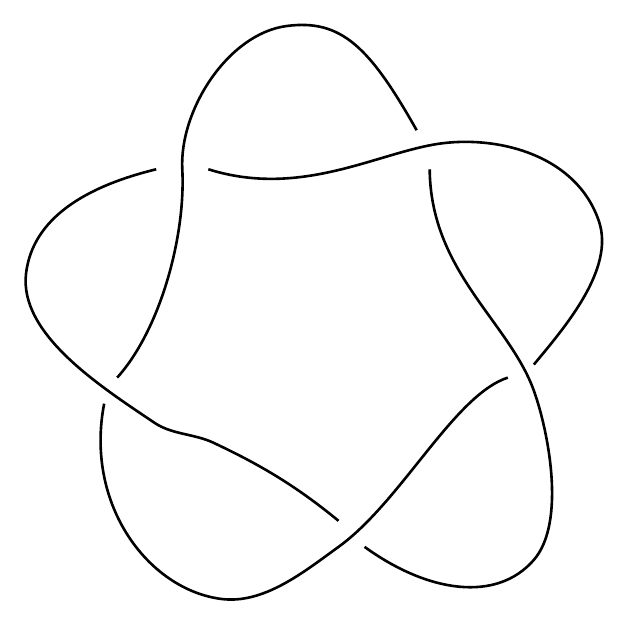}
\caption{The knot $5_1$}
\label{fig:5_1}
\end{center}
\end{figure}
\end{example}

By the Proposition \ref{prop1}, there is a subgroup of the automorphism group $Aut(S^{3},K)$ which acts on the fundamental quandle $Q(K)$ by quandle (anti)automorphisms. Every knot symmetry that preserves the peripheral structure induces a quandle (anti)automorphism. In the following, we offer some kind of a reverse to this correspondence. We need the important result of Waldhausen \cite{WA}. For the definition of an incompressible surface, we refer the reader to \cite[page 58]{WA}. 

\begin{definition} \label{def7} Let $M$ be a compact connected 3-manifold. 

$M$ is called \textbf{irreducible} if any smooth submanifold $S\subset M$ homeomorphic to a sphere bounds a subset $D$ which is homeomorphic to the closed 3-ball.

An irreducible manifold $M$, which is not a 3-ball, is called \textbf{sufficiently large} if it contains an incompressible surface. 

$M$ is called \textbf{boundary irreducible} if its boundary $\partial M$ is incompressible. 
\end{definition}

\begin{Theorem}[Waldhausen, Corollary 6.5 \cite{WA}] \label{th1} Suppose that $M$ and $N$ are irreducible and boundary irreducible. Let $M$ be sufficiently large and let $\psi \colon \pi _{1}(N)\to \pi _{1}(M)$ be an isomorphism preserving the peripheral structure. Then there exists a homeomorphism $f\colon N\to M$ inducing $\psi $.  
\end{Theorem}

In his pioneering work on quandles (which he called \textit{distributive grupoids}), Sergei Matveev proved the fact that the knot quandle is a complete knot invariant. Both our Theorem \ref{th3} and its proof are actually an adapted form of the Matveev's theorem \cite[Theorem 2]{MA}.   

\begin{Theorem} \label{th2} Let $F\colon Q(K_{1})\to Q(K_{2})$ be an isomorphism of the quandles of two nontrivial knots $K_{1}$ and $K_{2}$. Then there exists a homeomorphism $f\colon (S^{3},K_{1})\to (S^{3},K_{2})$ and some $g\in G_{K_{2}}$ such that $f_{*}=F^{g}$. The homeomorphism $f$ preserves the orientation of the normal bundle of $K_{1}$. 
\end{Theorem}
\begin{proof} Denote by $E_{K_{i}}=S^{3}-\textrm{Int}(N_{K_{i}})$ the complement of a regular neighbourhood of $K_{i}$. The manifold $E_{K_{i}}$ is irreducible and sufficiently large. Since $K_{i}$ is nontrivial, $E_{K_{i}}$ is also boundary irreducible. Remember that $z_{K_{i}}\in E_{K_{i}}, z'_{K_{i}}\in \partial N_{K_{i}}$ are the chosen basepoints and $s_{K_{i}}$ is the chosen path in $E_{K_{i}}$ from $z'_{K_{i}}$ to $z_{K_{i}}$. The action of $G_{K_{2}}$ on $Q(K_{2})$ is transitive by the Lemma \ref{lemma4}; thus there exists a $g\in G_{K_{2}}$ such that $F(\widehat{s}_{K_{1}})^{g}=\widehat{s}_{K_{2}}$. Define an isomorphism $\widehat{F}\colon Q(K_{1})\to Q(K_{2})$ by $\widehat{F}(a)=F(a)^{g}$ and it follows that $\widehat{F}(\widehat{s}_{K_{1}})=\widehat{s}_{K_{2}}$. By the Lemma \ref{lemma5}, the group $G_{K_{i}}$ is associated to $Q(K_{i})$ . Therefore, the isomorphism $\widehat{F}\colon Q(K_{1})\to Q(K_{2})$ induces a unique isomorphism $\phi \colon G_{K_{1}}\to G_{K_{2}}$ by the Lemma \ref{lemma3}. By the Lemma \ref{lemma4}, the peripheral subgroup $H_{K_{i}}\leq G_{K_{i}}$ equals the stabilizer subgroup of the element $\widehat{s}_{K_{i}}$, which is the same as the centralizer subgroup of $\partial \wh{s}_{K_{i}}=m_{K_{i}}$. Since $\widehat{F}(\widehat{s}_{K_{1}})=\widehat{s}_{K_{2}}$, the restriction $\phi |_{H_{K_{1}}}$ is an isomorphism of $H_{K_{1}}$ to $H_{K_{2}}$ and $\phi (m_{K_{1}})=m_{K_{2}}$. Thus $\phi \colon G_{K_{1}}\to G_{K_{2}}$ is an isomorphism, preserving the peripheral structure, and by the Waldhausen's theorem \ref{th1} there exists a homeomorphism $\overline{f}\colon E_{K_{1}}\to E_{K_{2}}$ inducing $\phi $. Since $\phi (m_{K_{1}})=m_{K_{2}}$, the map $\overline{f}$ extends to a homeomorphism $f\colon (S^{3},K_{1})\to (S^{3},K_{2})$ which preserves the orientation of the normal bundle. We have
\begin{xalignat*}{1}
& \partial \wh{F}(\wh{a})=\phi (\partial \wh{a})=[f\circ \partial a]=[f\circ \overline{a}\circ m_{a(0)}\circ a]=[\overline{f(a)}\circ m_{f(a(o))}\circ f(a)]=\partial [f\circ a]=\partial f_{*}(\wh{a})\;,
\end{xalignat*} and by the Lemma \ref{lemmax} it follows that $\wh{F}(\wh{a})=f_{*}(\wh{a})$ for any $\wh{a}\in Q(K_{1})$. The map $f$ induces the quandle isomorphism $\widehat{F}\colon Q(K_{1})\to Q(K_{2})$, thus $f_{*}=\widehat{F}=F^{g}$.  
\end{proof}   

\begin{corollary} \label{cor1} Let $K_{1}$ and $K_{2}$ be two nontrivial knots with isomorphic fundamental quandles. Then either $K_{1}$ and $K_{2}$ are equivalent knots, or the knot $K_{1}$ is equivalent to $rmK_{2}$. 
\end{corollary}
\begin{proof} By the Theorem \ref{th2}, there exists a homeomorphism $f\colon (S^{3},K_{1})\to (S^{3},K_{2})$ which fixes the orientation of the normal bundle. The orientation of a knot, coupled with the orientation of its normal bundle, gives an orientation of the ambient manifold $S^{3}$. It follows that if $f$ preserves the orientation of the knot, it must also preserve the orientation of $S^{3}$ and thus $K_{1}$ is equivalent to $K_{2}$. If, on the other hand, $f$ reverses the orientation of the knot, it must also reverse the orientation of $S^{3}$, and thus $K_{1}$ is equivalent to $rmK_{2}$. 
\end{proof}

\begin{proposition} \label{prop2} If there exists an antiautomorphism of the fundamental quandle of a nontrivial knot $K$, then either $K=rK$ or $K=mK$. 
\end{proposition}
\begin{proof} Let $F\colon Q(K)\to Q(K)$ be an antiautomorphism. By the Remarks \ref{rem1} and \ref{rem2}, $F$ defines an isomorphism between the fundamental quandles $Q(K)$ and $Q(K^{d})$. It follows by the Theorem \ref{th2} that there exists a homeomorphism $f\colon (S^{3},K)\to (S^{3},K^{d})$ that takes the orientation of the normal bundle of $K$ to the orientation of the normal bundle of $K^{d}$ (which is opposite to that of $K$). If $f$ preserves the orientation of $S^{3}$, then it must reverse the orientation of $K$ and thus $K=rK$. If, on the other hand, $f$ reverses the orientation of $S^{3}$, then it must preserve the orientation of $K$ and thus $K=mK$. 
\end{proof}

\begin{remark} \label{rem4} It follows from the Corollary \ref{cor1} and the Proposition \ref{prop2} that we cannot retrieve the information about the knot's orientation from its fundamental quandle. The knot quandle contains only the information about the orientation of the normal bundle of $K$, which defines the quandle operations. This is why we will use the term symmetry of a knot for any homeomorphism $(S^{3},K)\to (S^{3},K)$, often without knowing (or specifying) whether or not it preserves the orientation of $K$ or its ambient manifold $S^{3}$. 
\end{remark}

\begin{corollary} \label{cor2} Let $F\colon Q(K_{1})\to Q(K_{2})$ be an isomorphism of the quandles of two nontrivial knots $K_{1}$ and $K_{2}$, for which $F(\wh{s}_{K_{1}})=\wh{s}_{K_{2}}$. Then there exists a homeomorphism $f\colon (S^{3},K_{1})\to (S^{3},K_{2})$, preserving the orientation of the normal bundle, such that $f_{*}=F$. 
\end{corollary}
\begin{proof} If $F(\wh{s}_{K_{1}})=\wh{s}_{K_{2}}$, then following the proof of the Theorem \ref{th2} we may take $g=1$ and thus $\wh{F}=F=f_{*}$. 
\end{proof}

\begin{corollary} \label{cor3} Let $K$ be a nontrivial knot. For any element $h$ of the peripheral subgroup $H_{K}$, the inner automorphism $S_{h}$ is induced by a homeomorphism. 
\end{corollary}
\begin{proof} By the Lemma \ref{lemma4}, $H_{K}$ is the stabilizer subgroup of the element $\wh{s}_{K}$. The statement then follows from the Corollary \ref{cor2}.
\end{proof}

\begin{proposition} \label{prop3} Let $F\colon Q(K)\to Q(K)$ be an (anti)automorphism of the fundamental quandle of a nontrivial knot $K$. Then $F$ is induced by a homeomorphism that preserves the orientation of the normal bundle if $F$ is an automorphism, and reverses the orientation of the normal bundle if $F$ is an antiautomorphism. 
\end{proposition}
\begin{proof} Let $K_{1}=(K,o,\wh{s}_{K_{1}})$ denote the knot $K$ with a fixed orientation of the normal bundle, equipped with a fixed element $\wh{s}_{K_{1}}$. Let $s_{K_{2}}\subset E_{K}$ be a path that represents the element $F(\wh{s}_{K_{1}})\in Q(K)$. Let $K_{2}=(K,\pm o,\wh{s}_{K_{2}})$ denote the knot $K$ with the same orientation of the normal bundle if $F$ is an automorphism, and the reverse orientation if $F$ is an antiautomorphism, equipped with the element $\wh{s}_{K_{2}}=F(\wh{s}_{K_{1}})$.  Now $F\colon K_{1}\to K_{2}$ is a quandle isomorphism for which $F(\wh{s}_{K_{1}})=\wh{s}_{K_{2}}$. By the Corollary \ref{cor2}, there exists a homeomorphism $f\colon (S^{3},K)\to (S^{3},K)$ such that $f_{*}=F$. The homeomorphism $f$ preserves the orientation of the normal bundle if $F$ is an automorphism, and reverses the orientation of the normal bundle if $F$ is an antiautomorphism. 
\end{proof}

\end{section}

\begin{section}{Examples and calculations}\label{sec4}
 
\begin{figure}[h]
\labellist
\normalsize \hair 2pt
\pinlabel $a$ at 20 170
\pinlabel $b$ at 100 115
\pinlabel $c$ at  180 70
\pinlabel $d$ at 45 45 
\pinlabel $e$ at 150 170
\pinlabel $f$ at 150 60
\pinlabel $g$ at 40 130
\pinlabel $h$ at 10 20 
\pinlabel $i$ at 130 140
\endlabellist
\begin{center}
\includegraphics[scale=0.8]{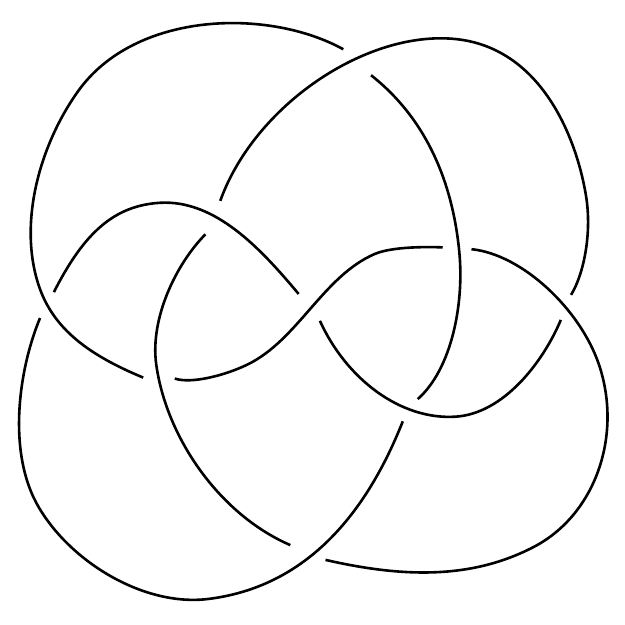}
\caption{The knot $9_{40}$}
\label{fig:9_40}
\end{center}
\end{figure}
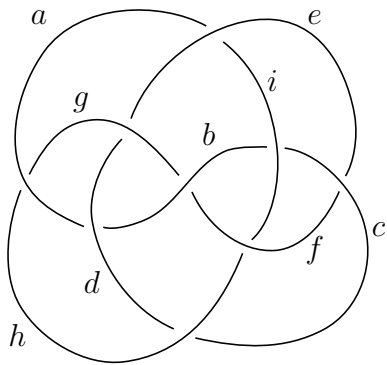
\begin{example} \label{ex5}
Consider the knot $9_{40}$ in the Rolfsen knot table, whose diagram is given on the Figure \ref{fig:9_40}. The presentation of its fundamental quandle is given by 
\begin{xalignat*}{1}
& Q(9_{40})=\langle a,b,c,d,e,f,g,h,i|\, i\tr e=a, h\tr a=g, b\tr d=a, f\tr b=g, c\tr i=b, \\
& f\tr c=e,c\tr h=d,e\tr g=d,i\tr f=h\rangle \;.
\end{xalignat*} 
Define a map $F\colon Q(9_{40})\to Q(9_{40})$ by $F(a)=g, F(b)=h, F(c)=i, F(d)=a,F(e)=b,F(f)=c,F(g)=d,F(h)=e,F(i)=f$. It is easy to check that $F$ preserves the crossing relations of the above presentation and thus defines a quandle automorphism of order 3. It follows by the Proposition \ref{prop3} that there exists a homeomorphism $f\colon (S^{3},9_{40})\to (S^{3},9_{40})$ of order 3. 
\end{example}

\begin{figure}[h]
\labellist
\normalsize \hair 2pt
\pinlabel $a$ at 180 70
\pinlabel $b$ at 80 185
\pinlabel $c$ at  0 60
\endlabellist
\begin{center}
\includegraphics[scale=0.6]{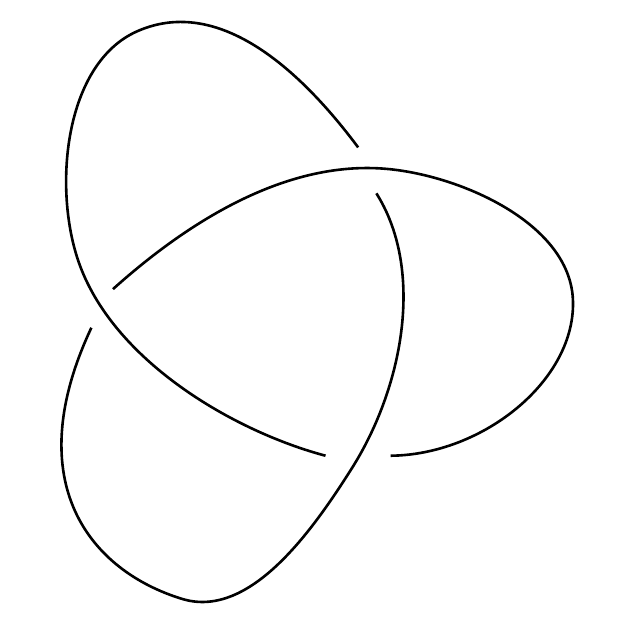}
\caption{The trefoil knot $3_{1}$}
\label{fig:3_1}
\end{center}
\end{figure}
\begin{example} \label{ex6} The trefoil knot $3_{1}$ has the following presentation of the fundamental quandle, obtained from the diagram in the Figure \ref{fig:3_1}: 
\[ Q(3_{1})=\left <a,b,c|\, a\tr b=c, b\tr c=a, c\tr a=b\right >\;.\] The corresponding presentation of the fundamental group $G_{K}$ is given by 
\begin{xalignat*}{1}
& G_{K}=\langle \partial a,\partial b,\partial c|\, (\partial b)^{-1}(\partial a)(\partial b)=\partial c, (\partial c)^{-1}(\partial b)(\partial c)=\partial a, (\partial a)^{-1}(\partial c)(\partial a)=\partial b\rangle \;.
\end{xalignat*}
Consider the element $g=(\partial a)(\partial b)=(\partial b)(\partial c)=(\partial c)(\partial a)\in G_{K}$. We calculate $a^{g}=a\tr b\tr c=c$, $b^{g}=b\tr c\tr a=a$ and $c^{g}=c\tr c\tr a=b$, so the inner automorphism $S_{g}$ is a quandle automorphism of order 3. It follows by the Proposition \ref{prop3} that the trefoil knot admits a symmetry of order 3.  
\end{example}

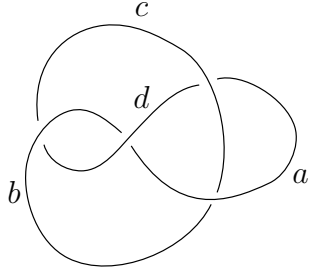
\begin{figure}[h]
\labellist
\normalsize \hair 2pt
\pinlabel $a$ at 180 70
\pinlabel $c$ at 80 175
\pinlabel $b$ at  0 60
\pinlabel $d$ at 80 120
\endlabellist
\begin{center}
\includegraphics[scale=0.6]{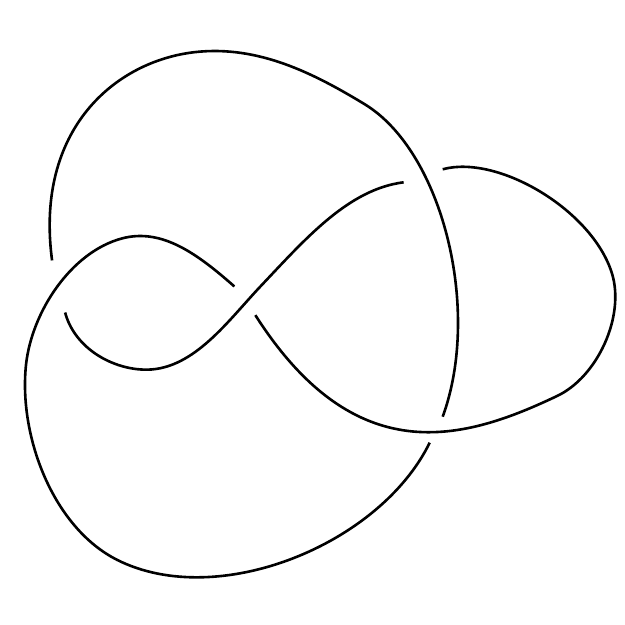}
\caption{The figure-eight knot $4_{1}$}
\label{fig:4_1}
\end{center}
\end{figure}
\begin{example} \label{ex7} The figure-eight knot $4_1$ has the following presentation of the fundamental quandle, obtained from the diagram in the Figure \ref{fig:4_1}: 
\[ Q(4_{1})=\left <a,b,c,d|\, c\tr a=b, a\tr d=b, a\tr c=d, c\tr b=d\right >\;.\]  Define a map $F\colon Q(K)\to Q(K)$ by $F(a)=c$, $F(b)=d$, $F(c)=a$ and $F(d)=b$. It is easy to see that $F$ preserves the crossing relations in the above presentation of $Q(4_1)$ and thus defines a quandle automorphism of order 2. It follows by the Proposition \ref{prop3} that the figure-eight knot admits a symmetry of order 2.  
\end{example}

By the Proposition \ref{prop3}, each (anti)automorphism of the fundamental quandle of a nontrivial knot is induced by a knot symmetry of $K$. Moreover, the existence of a quandle automorphism of order $p$ implies that $K$ admits a symmetry of order $p$. Thus, by observing the set of quandle (anti)automorphisms of $Q(K)$ we might learn something about the knot symmetries of $K$. Of course, the set $Aut^{C}(Q(K))$ of all quandle automorphisms is usually too large to grasp in its entirety. Instead, we might try to consider a simple subset of $Aut^{C}(Q(K))$ that is directly related to a given diagram of $K$.  

\begin{definition} \label{def7} Let $D_{K}$ be a diagram of a knot $K$ with an ordered set of arcs $(x_{1},\ldots ,x_{n})$, which defines a presentation $\langle x_{1},\ldots ,x_{n}|\, \textrm{ crossing relations }\rangle $ of the fundamental quandle $Q(K)$. A permutation $\sigma \in S_{n}$ is called a \textbf{$Q$-permutation} for $D_{K}$ if the map $x_{i}\mapsto x_{\sigma (i)}$ defines a quandle (anti)automorphism of $Q(K)$. The \textbf{$Q$-group} of the diagram $D_{K}$ is the subgroup of $S_{n}$ containing all the $Q$-permutations of $D_{K}$. 
\end{definition}

If $\sigma \in S_{n}$ is a $Q$-permutation, then the map $F_{\sigma }\colon Q(K)\to Q(K)$, defined on the generating set by $F_{\sigma }(x_{i})=x_{\sigma (i)}$, defines a quandle (anti)automorphism. By the Proposition \ref{prop3}, there exists a homeomorphism $f\colon (S^{3},K)\to (S^{3},K)$ such that $f_{*}=F_{\sigma }$. Therefore, every $Q$-permutation defines a knot symmetry of $K$, and the $Q$-group of $D_{K}$ corresponds to some subset of the automorphism set $Aut(S^{3},K)$. 

\begin{lemma} \label{lemma6} If $D_{K}$ is a knot diagram with $n$ arcs, then the order of the $Q$-group of $D_{K}$ is at most $2n$. 
\end{lemma}
\begin{proof} Let $(x_{1},\ldots ,x_{n})$ be the ordered set of arcs of the diagram $D_{K}$. Pick a crossing relation $x_{i}\tr x_{j}=x_{k}$. If $\sigma $ is a $Q$-permutation, then the image $\sigma (i)$ and the information about whether $F_{\sigma }$ defines an automorphism or an antiautomorphism of $Q(K)$ determine the images $\sigma (j)$ and $\sigma (k)$ and the images of all the other integers are determined by the crossing relations.
\end{proof}

A good thing about the $Q$-group is that it may be quite easily calculated from the planar representation of a knot diagram $D_{K}$. By the Lemma \ref{lemma6}, the order of the $Q$-group is bounded by twice the size of the knot diagram, so the time complexity of a suitable computer implementation should not be a problem. In the Subsection \ref{subs41} we provide a \verb|Phyton| code for the calculation of the $Q$-group from an alternating planar diagram.
  
The following table shows the calculations of the $Q$-groups of the minimal diagrams of some alternating knots. We used the diagrams given in the Rolfsen knot table at Knotatlas \cite{KAT}. 
\begin{center}
\begin{tabular}{|l|l|}
\hline
Knot $K$ & $Q$-group of $D_{K}$\\ \hline
$3_1$ & $D_{3}$\\ \hline
$4_1$ & $\ZZ _{4}$\\ \hline
$5_1$ & $D_{5}$\\ \hline
$5_2$ & $\ZZ _{2}$ \\ \hline
$6_1$ & $\ZZ _{2}$ \\ \hline
$6_2$ & $\{\rm{id}\}$ \\ \hline
$6_3$ & $\{\rm{id}\}$\\ \hline
$7_1$ & $D_{7}$ \\ \hline
$7_2$ & $\ZZ _{2}$ \\ \hline
$7_3$ & $\ZZ _{2}$ \\ \hline
$7_4$ & $\ZZ _{2}$\\ \hline
$9_{40}$ & $\ZZ _{6}$\\ \hline
\end{tabular}
\end{center}

\begin{subsection}{The Phyton code}\label{subs41}
\begin{verbatim}
from sympy.combinatorics.permutations import Permutation
from sympy.combinatorics.generators import symmetric
Permutation.print_cyclic=False
def quandlesym(PD):
    ### calculates the $Q$-group of an alternating knot diagram###
    L,P1,R=[],[],[]
    P=list(symmetric(len(PD)))
    for X in PD:
        L=L+[X[2],X[4]]
    for X in PD:
        if X[0] == 1:
            P1.append([L.index(X[1])//2,L.index(X[2])//2,L.index(X[3])//2])
        else:
            P1.append([L.index(X[3])//2,L.index(X[2])//2,L.index(X[1])//2])
    for i in range(1,len(P)+1):
        v1,v2 = True,True
        for Y in P1:
            if [P[i-1](Y[0]),P[i-1](Y[1]),P[i-1](Y[2])] not in P1:
                v1 = False
            if [P[i-1](Y[2]),P[i-1](Y[1]),P[i-1](Y[0])] not in P1:
                v2 = False
        if v1 or v2:
            R.append(P[i-1])
    return(R)
\end{verbatim}
\end{subsection}
\end{section}

\end{document}